\documentclass[11pt]{article}
\usepackage{graphicx,psfrag,amsmath,amsfonts,verbatim,mathtools}
\usepackage{times}
\usepackage{fancyhdr} 
\usepackage{enumerate}
\usepackage{enumitem} 
\usepackage[round]{natbib}
\numberwithin{equation}{section}

\usepackage{mathrsfs}
\usepackage{soul}

\usepackage{listings}

\usepackage{setspace}
\usepackage[perpage]{footmisc}
\usepackage[colorlinks,linkcolor=blue,urlcolor=blue]{hyperref}
\usepackage{placeins}
\usepackage{float}

\usepackage{graphicx}
\usepackage{subfigure}
\usepackage{epstopdf}
\usepackage{float}

\usepackage{geometry}
\usepackage{algorithm}
\usepackage{algorithmic}

\usepackage{endnotes}

\usepackage{multirow}
\usepackage{booktabs}
\usepackage{makecell}
\usepackage{footnotebackref}
\usepackage{tablefootnote}
\usepackage{diagbox}

\usepackage{breakcites}

\usepackage{amsmath,amsfonts,bm}
\usepackage{amssymb}
\usepackage[amsmath,thmmarks,framed]{ntheorem}
\theoremstyle{plain}
\usepackage{framed}
\usepackage{color}
\definecolor{gray}{rgb}{0.6,0.6,0.6}

\newtheorem{defn}{Definition}[section]
\newtheorem{thm}{Theorem}[section]
\newtheorem{prop}{Proposition}[section]

\newtheorem{lm}{Lemma}[section]
\newtheorem{assume}{Assumption}[section]
\newtheorem{remark}{Remark}[section]
\newtheorem{cor}{Corollary}[section]

\usepackage{xcolor}
\usepackage{framed}
\usepackage{lipsum}
\colorlet{shadecolor}{blue!20}

\usepackage{mdframed}

\usepackage{macros}

\begin{document}

\title{On the Convergence Rate of Stochastic Mirror Descent for Nonsmooth Nonconvex Optimization}
\author{
	Siqi Zhang\thanks{Department of Industrial and Enterprise Systems Engineering (ISE), University of Illinois at Urbana-Champaign (UIUC), Urbana, IL 61801, USA. Emails:
	\href{mailto:siqiz4@illinois.edu}{siqiz4@illinois.edu}, 
	\href{mailto:niaohe@illinois.edu}{niaohe@illinois.edu}. This work is supported by NSF CCF-1755829.
	}
	\and
	Niao He\footnotemark[1]
}
\date{}
\maketitle

\begin{abstract}
	In this paper, we investigate the non-asymptotic stationary convergence behavior of Stochastic Mirror Descent (SMD) for nonconvex optimization. We focus on a general class of nonconvex nonsmooth stochastic optimization problems, in which the objective can be decomposed  into a relatively weakly convex function (possibly non-Lipschitz) and a simple non-smooth convex regularizer.  We prove that SMD, without the use of mini-batch,  is guaranteed to converge to a stationary point in a convergence rate of $ \mathcal{O}(1/\sqrt{t}) $. The efficiency estimate matches with existing results for stochastic subgradient method, but is evaluated under a stronger stationarity measure. Our convergence analysis applies to both the original SMD and its proximal version, as well as the deterministic variants, for solving relatively weakly convex problems. 

\end{abstract}

\section{Introduction}

In this paper, we consider the composite nonsmooth nonconvex stochastic optimization problems with the following general form
\begin{equation}\label{eq:obj}
\min_{x\in{X}}\;T(x):= f(x)+r(x)=\EE_{\xi}\left[F(x;\xi)\right]+r(x)
\end{equation} 
where $ {X}\subseteq\mathcal{X} $ is a nonempty closed convex subset of a finite-dimensional Euclidean space $ \mathcal{X} $ equipped with a norm $ ||\cdot|| $, $ f(x):{X}\rightarrow\RR $ is a nonsmooth nonconvex function, $ r(x):{X}\rightarrow\RR $ is a simple nonsmooth convex regularizer. 

Throughout, we assume that $f(x)$ is \emph{$ \rho $-relatively weakly convex}, \ie, the function $f(x)+\rho\omega(x)$ is convex for some $\rho>0$ and some function $\omega(x):{X}\to\RR$ that is continuously differentiable and $1$-strongly convex with respect to the norm $\|\cdot\|$ defined on $\Xcal$. In the case when $\omega(x)=\frac{1}{2}\|x\|^2$ and $\|\cdot\|$ is an inner product induced norm, $f(x)$ is also called \emph{$\rho$-weakly convex}. Weak convexity is a special yet very common case of nonconvex functions, which contains all convex functions and Lipschitz smooth functions. The composite form of the optimization problem covers a wide spectrum of regularized problems in machine learning, including the nonlinear least square,  sparse logistic regression \citep{jun2009largescale,xinyue2018nonconvex}, sparse recovery~\citep{laming2014convergence}, and robust phase retrieval \citep{davis2017nonsmooth}. 

When the function $f(x)$ is convex, the proximal variant of Stochastic Mirror Descent (SMD) is one of the most widely used algorithms for solving the above composite problem; see, e.g.,~\cite{duchi2010composite},~\cite{he2015saddle}, and~\cite{beck17firstorder}. SMD performs the recurrence at each iteration: 
\beq{alg:SMD}
x_{t+1}=\argmin_{x\in {X}}\;\{\la F'(x_t,\xi_t),x\ra+r(x) +\frac{1}{\alpha_t}D_\psi(x,x_t)\}
\eeq
where $\alpha_t>0$ is the stepsize, $F'(x_t,\xi_t)$ is an unbiased estimator of the subgradient of $f(x)$ at $x_t$, the term $D_\psi(x,x_t):=\psi(x)-\psi(x_t)-\la\nabla\psi(x_t),x-x_t\ra$ stands for some Bregman divergence generated by a $1$-strongly convex and continuously differentiable function $\psi(x)$ defined on $X$. Note that when the Bregman divergence is set to be the simple Euclidean distance, \ie, $D_\psi(x,x')=\frac{1}{2}\|x-x'\|_2^2$ with $\psi(x)=\frac{1}{2}\|x\|_2^2$, SMD reduces to proximal stochastic subgradient method (SGD). When there is no regularizer, \ie, $r(x)=0$, this reduces to the original SMD~\citep{blair1985problem}. The non-asymptotic convergences of SMD algorithm and its variants have been extensively studied in the convex regime; see e.g., \cite{Nem09} for analysis of the original SMD, and \cite{duchi2010composite} for the proximal SMD.  It is well-known that SMD achieves an optimal $\Ocal(1/\sqrt{t})$ convergence rate for solving general composite nonsmooth convex problems with unimprovable constant factors.  However, the non-asymptotic convergence behavior of SMD is far from fully understood when moving to the nonconvex regime.

\subsection{Related Works}
There have been several recent works discussing the convergences of SGD or SMD for nonconvex problems. We mainly focus on the purely stochastic case, where $f(x)$ can only be accessed through stochastic oracles. The special case where $f(x)$ consists of a finite sum of components is beyond the scope of this work. 

The seminal work by \cite{ghadimi2013stochastic} provides the first non-asymptotic convergence analysis of SGD for unconstrained smooth nonconvex objectives, \ie, problem \eq{eq:obj} with ${X}=\RR^n, r(x)=0$ and smooth $f(x)$. They show that a modified SGD, called \emph{Randomized Stochastic Gradient (RSG)}, requires $\Ocal(1/\epsilon^2)$ number of iterations to generate an $\eps$-stationary point such that $\EE[\|\nabla f(x)\|_2^2]\leq\eps$. Later, \cite{ghadimi2016mini} addressed the general constrained composite problem \eq{eq:obj} with smooth $f(x)$ and proposed a modified mini-batch SMD method, called \emph{Randomized Stochastic Projected Gradient (RSPG)}, that requires using a mini-batch of size $\Ocal(1/\epsilon)$ samples to estimate the gradient at each iteration. They showed that the algorithm achieved  the same $ \mathcal{O}(1/\epsilon^2) $  overall sample complexity to achieve an $\epsilon$-stationary point, in terms of the generalized projected gradient, \ie, $\EE[\|g_{X}(x)\|_2^2]\leq\eps$ \footnote{Here the generalized project gradient is defined as $g_X(x_t):=\frac{1}{\alpha_t}(x_t-x_{t+1})$, where $x_{t+1}$ is defined in \eq{alg:SMD}.}. It is worth mentioning that although the RSPG algorithm utilizes the mirror descent framework, the analysis in \cite{ghadimi2016mini} only applies to the Euclidean setting and smooth objectives. 

To overcome the mini-batch requirement for solving constrained nonconvex problems, \cite{davis2017proximally} proposed the \emph{Proximally Guided Stochastic subGradient (PGSG)} method, by combining proximal point algorithm  and SGD in a nested framework - iteratively solving subproblems arising from proximal point algorithm through SGD routines. The algorithm solves problem \eq{eq:obj} with $r(x)=0$ and $\rho$-weakly convex $f(x)$,  and  attains the $ \widetilde{\mathcal{O}}(1/\eps^2) $ \footnote{ $ \widetilde{\mathcal{O}}(\cdot) $ means the complexity neglects its logarithmic terms in its expression.} sample complexity to get an $\eps$-stationary point, measured by the squared distance from zero to the Fr\'echet subdifferential set, \ie, $\EE\left[\text{dist}^2\left(0,\partial_F(f+\delta_{X})(x)\right)\right]\leq \eps$.  More recently, \cite{davis2018stochastic} considered the same weakly convex setting and showed that even the basic SGD and its proximal variant converge and exhibit an $ \mathcal{O}(1/\eps^2) $ sample complexity. 

\begin{table}[htbp]
	\centering 
	\footnotesize
	\renewcommand\arraystretch{1.4}
	\caption{Summary of algorithms}
	\label{tab:comparison}
	\begin{tabular}{l*{3}{|m{2.5cm}<{\centering}}*{1}{|m{3cm}<{\centering}}*{1}{|m{1.8cm}<{\centering}}}
		\toprule[1.5pt]
		
		{Method}
		& \emph{RSG} \hspace{1cm} \citep{ghadimi2013stochastic}
		& \emph{RSPG}\hspace{1cm} \citep{ghadimi2016mini}
		& \emph{PGSG}\hspace{1cm} \citep{davis2017proximally}
		& \emph{PSG} \citep{davis2018stochastic}
		& \emph{SMD}\hspace{2cm} (this paper)
		\\
		\midrule[1.5pt]
		
		{Convexity}
		& NC 
		& NC + C
		& WC
		& WC + C
		& RWC + C
		\\
		\hline
		
		{Smoothness}
		& Lip-smooth
		& Lip-smooth
		& Lip-continuous
		& Lip-continuous
		& (relative) Lip-continuous
		\\
		\hline
		
		{Constraint} 
		& $ \mathbb{R}^n $
		& closed convex
		& closed convex
		& $ \mathbb{R}^n $
		& closed convex
		\\
		\hline
		
		{Stationary} 
		& $ \mathbb{E}[||\nabla f(x)||_2^2] $ 
		& $ \mathbb{E}[||g_{X}(x)||_2^2] $
		& $ \mathbb{E}[||\mathcal{G}_{1/(2\rho)}(x)||^2] $
		& $ \mathbb{E}[||\mathcal{G}_{1/(2\rho)}(x)||^2] $
		& $ \mathbb{E}[\Delta_{1/(2\rho)}(x)] $
		\\
		\hline
		
		
		{Complexity}
		& $ \mathcal{O}(1/\epsilon^2) $
		& $ \mathcal{O}(1/\epsilon^2) $
		& $ \widetilde{\mathcal{O}}(1/\epsilon^2) $
		& $ \mathcal{O}(1/\epsilon^2) $
		& $ \mathcal{O}(1/\epsilon^2) $
		\\
		\hline
		
		{Batch size} 
		& 1
		& $ \mathcal{O}(1/\epsilon) $
		& 1 (per inner iter)
		& 1
		& 1
		\\
		\hline
		
		{Setting} 
		& Euclidean
		& Euclidean
		& Euclidean
		& Euclidean
		& Non-Euclidean
		\\
		\hline
		
		\multicolumn{6}{c}{C = Convex, NC = Nonconvex, WC = Weakly Convex, RWC = Relatively Weakly Convex} \\
		\multicolumn{6}{c}{Lip-smooth = Lipschitz Smooth, Lip-continuous = Lipschitz Continuous} \\
		\multicolumn{6}{c}{ $g_{X}(x_t)=\alpha_t^{-1}\|x_t-x_{t+1}\|$; $ \mathcal{G}_{1/(2\rho)}(x)$ and $\Delta_{1/(2\rho)}(x)$ are defined in \eq{eq:measure} and \eq{eq:measure_2}}\\
		\bottomrule[1.5pt]
	\end{tabular}
\end{table}

However, none of these works have considered or addressed the convergence behavior of SMD in the non-Euclidean setting. We point out that a recent work by \cite{zhou2017stochastic} investigated the asymptotic convergence of SMD, but is only limited to a very special class of nonconvex problems that ensures global convergence. This paper aims to close this fundamental theoretical gap and establish the non-asymptotic stationary convergence analysis of Stochastic Mirror Descent for nonconvex problems.  A detailed comparison of this work and previous ones is summarized in Table~\ref{tab:comparison}.

\subsection{Contribution}
In this paper, we establish the non-asymptotic stationary convergence rate analysis of SMD for constrained stochastic composite optimization problems in the general form of \eq{eq:obj}, where the objective is $\rho$-relatively weakly convex.  We consider SMD with  distance generating function setting to be exactly $\psi(x)=\omega(x)$. Our results apply to the basic SMD and its proximal version as well as the deterministic variants. More specifically, the main contributions can be summarized as follows. 

Firstly, inspired by \cite{davis2018stochastic}, we construct a new measure of stationary convergence called \emph{Bregman gradient mapping} based on the \emph{Bregman proximal operator}: 
\beq{eq:measure}
\Gcal_\lambda(x):=\frac{1}{\lambda}(x-\prox_{\lambda T}(x)) 
\eeq 
where the Bregman proximal operator $\prox_{\lambda T}(x):=\argmin_{y\in{X}}\{T(y)+\frac{1}{\lambda}D_\omega(y,x)\}$. When $T(x)$ is $\rho$-relatively weakly convex, the stationary measure is well-defined as long as $\lambda<\rho^{-1}$. Note that this is very distinct from the notion of generalized projection gradient used in \cite{ghadimi2016mini}. We also define another measure induced by Bregman divergence, called  \emph{Bregman stationarity}, 
\beq{eq:measure_2}
\Delta_\lambda(x):=\frac{1}{\lambda^2}\cdot\left(D_\omega(x,\prox_{\lambda T}(x))+D_\omega(\prox_{\lambda T}(x),x)\right).
\eeq
This quantity provides a stronger convergence criterion since $\|\Gcal_\lambda(x)\|^2\leq \Delta_\lambda(x)$. When the distance generating function $\omega(x)=\frac{1}{2}\|x\|_2^2$, both stationary measures reduce to the one used in \cite{davis2017proximally} and \cite{davis2018stochastic} based on the gradient of the Moreau envelope of the objective. Later, we provide detailed analysis of these stationary measures and its relations to the gradient of  Bregman Moreau envelope and traditional stationary measure, \ie, $\text{dist}(0,\partial (T+\delta_{X})(x)$ in this problem.

As a main result, we show that SMD converges to a $\eps$-stationary point such that $\EE[\Delta_{1/(2\rho)}(x)]\leq\epsilon$ within $\mathcal{O}(1/\epsilon^2)$ iterations. The rate matches with that of stochastic subgradient method recently established in~\cite{davis2018stochastic} and implies that using mini-batch is not necessary for SMD to converge for relatively weakly convex problems. 
This appears to be the first non-asymptotic convergence result for SMD in the nonconvex, nonsmooth regime, to the best of our knowledge. We provide a unified and simplified convergence analysis that apply to both plain SMD and its proximal variant. In contrast, \cite{davis2018stochastic} requires different analysis for the projected and proximal versions of stochastic subgradient method.

Lastly, we extend these results to a much weaker condition by assuming only relative continuity of the objective function. We show that similar convergence result  can be obtained under this relaxed assumption. 



\subsection{Paper Organization}
The paper is organized as follows. In Section~\ref{sec:prelim}, we  introduce the concepts of relative weak convexity and Bregman stationarity measures. We also review some important properties of Bregman Moreau envelope and Bregman proximal operator. In Section~\ref{sec:main}, we present the SMD algorithm and its stationary convergence guarantee.  Finally, in Section~\ref{sec:extensions}, we further extend the results to relative Lipschitz continuous problems.

\section{Relatively Weak Convexity and Bregman Stationarity}\label{sec:prelim}
In this section, we  first introduce the concept of relatively weakly convex functions and discuss some important properties and calculus of this family of nonconvex functions.  

\subsection{Relatively Weakly Convex Functions}


Let $X$ be a closed convex set and $\mathcal{X}$ be its embedding Euclidean space associated with some norm $\|\cdot\|$. Let $\omega(x):X\to\RR$ be a reference function that is continuously differentiable and 1-strongly convex on ${X}$ with respect to the given norm $\|\cdot\|$, \ie, $\omega(x)-\omega(y)-\langle\nabla\omega(y),x-y\rangle\geq\frac{1}{2}\|x-y\|^2$ for any $ x,\ y\in X $. This induces the \emph{Bregman divergence}, denoted by $ D_{\omega}(x, y) $:
\begin{equation}
D_{\omega}(x, y) = \omega(x)-\omega(y)-\langle\nabla\omega(y),x-y\rangle. 
\end{equation}
It follows immediately that $D_{\omega}(x, y)\geq\frac{1}{2}\|x-y\|^2$. 

\begin{defn}
	\emph{(Relatively Weak Convexity)} 
	 A function $f(x):X\to\RR$ is said to be $\rho$-relatively weakly convex on $X$ with respect to the reference function $\omega(x)$ if $f(x)+\rho\omega(x)$ is convex on $X$. 
	We denote $f(x)$ as $(\rho,\omega(\cdot))$-RWC. 
\end{defn}

The above definition generalizes the traditional notion of weak convexity introduced in \cite{vial1983strong} and extensively studied in existing works~\citep{drusvyatskiy2017proximal,davis2017proximally,davis2018stochastic}.  In the  case when $\omega(x)=\frac{1}{2}\|x\|^2$ with some inner product induced norm $\|\cdot\|$,  such as the Euclidean norm, the function $f(x)$ is called $\rho$-weakly convex. Obviously, any $\rho$-weakly convex function is also $(\rho,\omega(\cdot))$-weakly convex,  for any reference $\omega(x)$ that is $1$-strongly convex with respect to the norm $\|\cdot\|$. However, the class of relatively weakly convex functions can be much broader. For example, the function $f(x)=-\sum_{i=1}^nx_i\log(x_i)$ is relatively weakly convex, but not weakly convex.

In what follows, we will provide some equivalent characterizations of relatively weakly convexity. 


\begin{prop}\label{prop:RWC}
Let $X\subseteq U$, where $U$ is a convex open set. 
The following statements are equivalent:
\begin{enumerate}
	\item[(i)] $f(x)$ is $(\rho,\omega(\cdot))$-RWC on $U$.
	\item[(ii)] For any fixed $ y\in U$,  $ f_\rho(x;y)\coloneqq f(x)+\rho D_\omega(x,y) $ is convex in $ x\in U $.
	\item[(iii)] For any fixed $ y\in U $, there exists $g\in\Xcal$, such that
	\begin{equation}\label{eq:wc}
	f(x)\geq f(y)+\langle g,x-y \rangle-\rho D_\omega(x,y), \forall x\in U.
	\end{equation}
\end{enumerate}
\end{prop}

\begin{proof} (i)$\Rightarrow$ (ii) is straightforward. $(\rho,\omega(\cdot))$-RWC implies that $ f(x)+\rho\omega(x) $ is convex. Hence, 
	\begin{equation*}
		f_\rho(x;y)=f(x)+\rho \left[\omega(x)-\omega(y)-\langle\nabla\omega(y),x-y\rangle\right]=\left[f(x)+\rho\omega(x)\right]-\left[\rho\omega(y)+\langle \nabla\omega(y),x-y\rangle\right]
	\end{equation*}
	is equal to the sum of a convex function and an affine function, thus convex. (ii)$\Rightarrow$ (iii) is also straightforward. Since $f_\rho(x;y)$ is convex in $x\in U$, subgradients exist on $U$. Let $g\in\partial f_\rho(y;y)$. We have
	\begin{equation*}
		f(x)+\rho D_\omega(x,y)\geq f(y)+\rho D_\omega(y,y)+\langle g, x-y\rangle, \forall x\in U
	\end{equation*}
	Rearranging the terms,  we obtain the third statement. Lastly, we show (iii)$\Rightarrow$ (i). Invoking the definition of Bregman divergence, \eq{eq:wc} implies that for any $y\in U$, there exists $g\in\Xcal$
	\begin{eqnarray*}
		&&f(x)\geq f(y)+\langle g, x-y\rangle-\rho\left[\omega(x)-\omega(y)-\langle\nabla\omega(y),x-y\rangle\right], \forall x\in U\\
		&\Leftrightarrow& [f(x)+\rho\omega(x)]\geq [f(y)+\rho\omega(y)]+\la g+\rho\nabla\omega(y),x-y \ra,\forall x\in U
	\end{eqnarray*}
	This implies that $f(x)+\rho \omega(x)$ is convex, \ie, $f(x)$ is $(\rho,\omega(\cdot))$-RWC.
\end{proof}

In fact, the above results also provide a valid subdifferential set of relatively weakly convex functions and the construction of subgradients. 

\begin{defn}\label{def:RWC}\emph{(Subgradient and Subdifferential Set)} A vector $g\in\Xcal$ is a \emph{subgradient} of $f(x)$ at $x\in X$ if $g\in\partial f_\rho(x;x)$. The subdifferential set of $f(x)$ at $x$, denoted as $\partial f(x)$, contains all subgradients at $x$. Note that $\partial f(x)=\partial (f+\rho\omega)(x)-\rho\nabla\omega(x)$.
\end{defn}
For $\rho$-relatively weakly convex functions, the above subdifferential set is always well-defined and non-empty. When the function is weakly convex (thus locally Lipschitz), this set is also equivalent to the Fr\'echet subdifferential set and the Clarke differential set~\citep{davis2017proximally}.   

\paragraph{Examples.}A major class of relatively weakly convex functions is the family of smooth functions with Lipschitz continuous gradients.  Suppose $f(x)$ is continuously differentiable and has $\rho$-Lipschitz continuous gradient, \ie, $||\nabla f(x)-\nabla f(y)||_*\leq \rho||x-y||$, where $\|\cdot\|_*$ is the dual norm of $\|\cdot\|$, then by fundamental theorem of calculus, this implies that
$|f(x)-f(y)-\langle\nabla f(y),x-y\rangle|\leq\frac{\rho}{2}||x-y||^2$. 
Hence, it follows 
$$f(x)\geq f(y)-\langle\nabla f(y),x-y\rangle-\frac{\rho}{2}||x-y||^2\geq f(y)-\langle\nabla f(y),x-y\rangle-\rho D_\omega(x,y), 
$$
so $f(x)$ is relatively weakly convex. 
In the case when both $f(x)$ and $\omega(x)$ is twice differentiable, relative weak convexity is equivalent to say $\nabla^2 f(x)\succeq -\rho\nabla^2\omega(x)$. Hence, the family of relatively weakly convex functions also include functions that are not necessarily Lipschitz smooth, e.g., the relatively smooth functions~\citep{lu2018relatively}. Moreover, the following proposition gives some calculus and more examples of relatively weakly convex functions.
\begin{prop}
\label{prop:calculus} 
Let $X$ be a nonempty closed convex set. 
\begin{enumerate}
\item[(a)] Suppose $f_1:X\to\RR$ is $(\rho_1,\omega_1(\cdot))$-RWC and $f_2:X\to\RR$ is $(\rho_2,\omega_2(\cdot))$-RWC on $X$, then $f_1+f_2$ is $(\rho_1+\rho_2,\bar{\omega}(\cdot))$-RWC on $X$, where $\bar\omega(x)=(\rho_1+\rho_2)^{-1}(\rho_1\omega_1(x)+\rho_2\omega_2(x))$ is differentiable and $1$-strongly convex on $X$. 
\item[(b)] Suppose $f_i:\RR^n\to\RR$ is $(\rho_i,\omega(x))$-RWC for $i\in I$, and $\rho:=\sup_{i\in I}\rho_i<\infty$, then the supreme function $f(x):=\sup_{i\in I}f_i(x)$ is also $(\rho,\omega(x))$-RWC.
\item[(c)] Suppose $f:\RR^d\to\RR$ is closed convex and $L_f$-Lipschitz continuous such that $|f(u)-f(v)|\leq L_f \|u-v\|,\forall u,v\in\RR^d$,  and suppose $g:\RR^n\to\RR^d$ is $L_g$-relatively smooth with respect  to $\omega(x)$ such that for any $x,y\in\RR^n$
$$\|g(x)- g(y)-\la \nabla g(y),x-y\ra\|\leq L_g\cdot D_\omega(x,y).$$ Then the composition 
$f\circ g:\RR^n\to\RR$ is $(L_fL_g, \omega(\cdot))$-RWC. 
\end{enumerate}
\end{prop}
\begin{proof}
Parts (a) and (b) of Proposition~\ref{prop:calculus} are straightforward. Part (c) of Proposition~\ref{prop:calculus} is because that for any $x,y$, $w\in\partial f(g(y))$, it holds that 
\begin{eqnarray*}
f\circ g(x)&\geq& f\circ g(y)+\la w, g(x)-g(y)\ra\\
&=& f\circ g(y)+\la w, \nabla g(y)^T(x-y)\ra+\la w, g(x)-g(y)-g(y)^T(x-y)\ra\\
&\geq& f\circ g(y)+\la w, \nabla g(y)^T(x-y)\ra -L_g \|w\|_*\cdot D_\omega(x,y)\\
&\geq& f\circ g(y) + \la \nabla g(y)w, x-y\ra -L_fL_g D_\omega(x,y)
\end{eqnarray*}
The first inequality is due to the convexity of $f$; the third inequality is due to H\"older’s inequality; and the last inequality is due to the Lipschitz continuity of $f$. 
\end{proof}

\subsection{Bregman Moreau Envelope and Bregman Proximal Operator}
We now revisit the basic properties of Bregman divergence and introduce the stationary measures based on Bregman Moreau envelope. We first  list a few important properties of Bregman divergence that will be heavily used in the rest of the paper. 
\begin{lm}\label{lem:threepoint}
\emph{(Properties of Bregman Divergence, Section~9.2.1, \cite{beck17firstorder})}

	\begin{enumerate}[labelindent=1em,labelwidth=\widthof{\ref{last-item}},leftmargin=!]
		\item[(a)] The Bregman divergence satisfies the three-point identity:
		\begin{equation}\label{threeProII}
		D_{\omega}(x,y)+D_{\omega}(y,z)=D_{\omega}(x,z)+\langle\nabla\omega(z)-\nabla\omega(y),x-y\rangle, \forall x,y,z\in X
		\end{equation} 
		\item[(b)] Suppose $\phi(x)$ is convex and $ z^+=\underset{x\in{X}}{\argmin}\{\phi(x)+\frac{1}{\alpha}D_{\omega}(x,z)\}$ for some $\alpha>0$, then we have 
		\begin{equation}\label{threeProI}
		\phi(x)+\frac{1}{\alpha}D_{\omega}(x,z)\geq\phi(z^+)+\frac{1}{\alpha}D_{\omega}(z^+,z)+\frac{1}{\alpha}D_{\omega}(x,z^+),  \forall\ x\in{X}.
		\end{equation}
	\end{enumerate}
\end{lm}

Below we provide the definitions of Bregman Moreau envelope and Bregman proximal operator, which are natural extensions of Moreau envelope and proximal operator by replacing Euclidean distance with Bregman divergence~\citep{bauschke2006joint}.
Because of the asymmetry of Bregman divergence, we should be careful when extending Moreau envelope directly to the Bregman case. We consider the \emph{(left) Bregman envelope} and the \emph{(left) Bregman proximal operator} here\footnote{In fact, there are also ``right" versions of the Bregman envelope and proximal operator, with some different properties~\citep{bauschke2006joint}, but here we will focus on the left version.}.

\begin{defn} \emph{(Bregman Moreau envelope and proximal operator)} Given positive number $\lambda>0$ and a function $T(x)$, for a vector $z\in X$, we define its Bregman Moreau envelope as 
\begin{equation}
T_{\lambda}(z):
=\min_{x\in{X}}\left\{T(x)+\frac{1}{\lambda}D_{\omega}(x,z)\right\}
\end{equation}
and the corresponding Bregman proximal operator
\begin{equation}
\prox_{\lambda T}(z)
:=\underset{x\in{X}}{\argmin}\left\{T(x)+\frac{1}{\lambda}D_{\omega}(x,z)\right\}
\end{equation}
\end{defn}

 It is obvious that when the function $T(\cdot)$ is convex, then the proximal operator is always well-defined and unique for any positive number $\lambda>0$. In fact, this holds true for any $(\rho,\omega(\cdot))$-RWC functions as long as $0<\lambda< \rho^{-1}$. More specifically, we have 
\begin{lm}\label{lem:uniqueness}
\emph{(Uniqueness of Bregman proximal operator)} Suppose a function $T(x)$ is $(\rho,\omega(\cdot))$-RWC on $X$ and $0<\lambda<\rho^{-1}$. Then for any input $z\in X$,  the function $T(x)+\frac{1}{\lambda}D_\omega(x,z)$ is  $(\lambda^{-1}-\rho)$-strongly convex. Moreover, the Bregman proximal operator $ \prox_{\lambda T}(z) $ is unique.
\end{lm}
The result follows directly from the definition of relative weak convexity. Same as the Euclidean case, one can also show that the Bregman Moreau envelope is differentiable. 
	

\begin{lm}\label{lem:gradient}
	\emph{(Gradient of Bregman Moreau envelope)} Suppose $T(x)$ is a proper closed function and $(\rho, \omega(\cdot))$-RWC on $X$, and  and $0<\lambda<\rho^{-1}$. Suppose the above DGF $\omega(x)$ is also twice continuously differentiable.  Then the Bregman Moreau envelope $T_\lambda(z)$ is differentiable, and its gradient is given by 
	\begin{equation}
	\nabla T_{\lambda}(z) =\frac{1}{\lambda}\nabla^2\omega(z)(z-\prox_{\lambda T}(z)).
	\end{equation}
\end{lm}
The result follows immediately from Propositions 3.10 and 3.12 in \cite{bauschke2006joint} by using the convexity of $\lambda T(x)+\omega(x)$. For sake of simplicity, we do not repeat the details here.

\subsection{Stationarity Measures}

Since the major goal of solving a general nonsmooth nonconvex problem is to find a stationary point, we are mainly interested in analyzing the stationary convergence of the SMD algorithm. For the general constrained composite problem in the form of \eq{eq:obj},  a stationary point $x^*$ can often be described as such that $0\in\partial (T+\delta_X)(x^*)$, or equivalently, $\text{dist}_{\|\cdot\|}(0,\partial (T+\delta_X)(x^*))=0$. Here we use $ \text{dist}_{\|\cdot\|}(x,Q)\coloneqq\text{inf}_{y\in Q}||y-x|| $ to characterize the distance between a point $x\in X$ and a set $Q$ under a specific norm $\|\cdot\|$ and we use $ \delta_X(\cdot) $ to denote the indicator function of the set $ X $.  

Inspired by~\cite{davis2018stochastic}, a natural option to measure the stationarity of a candidate solution $x\in X$ is by evaluating the difference of $x$ and its proximity: 
\beq{eq:measure2}
\Gcal_\lambda(x):=\frac{1}{\lambda}(x-\prox_{\lambda T}(x)), \text{ where } \lambda\in(0,\rho^{-1}).
\eeq 
In the case when the DGF is $\omega(x)=\frac{1}{2}\|x\|^2_2$, it follows immediately from Lemma~\ref{lem:gradient} that $\Gcal_\lambda(x)=\nabla T_{\lambda}(z)$, \ie, the gradient of the Moreau envelope of $T(x)$~\citep{davis2018stochastic}. Under such a case, invoking the definition of proximal operator, and denoting $ \hat{x}\coloneqq\prox_{\lambda T}(x) $, we have $ \Gcal_\lambda(x)\in\partial (T+\delta_X)(\hat{x}) $, so one can show that 
\beq{eq:distance}
\text{dist}^2_{\|\cdot\|_2}(0,\partial (T+\delta_X)(\hat{x}))\leq \|\Gcal_\lambda(x)\|_2^2=\|\nabla T_{\lambda}(x)\|_2^2.
\eeq
Hence, the magnitude of $ \Gcal_\lambda(x) $ provides an upper bound for the distance from the origin to the subdifferential set $\partial f(x)$, and can be used to measure the progress of iterations. In our case, for general choices of distance generating functions $\omega(x)$, this also makes sense. From Lemma~\ref{lem:gradient}, suppose $\omega(x)$ is twice continuously differentiable, we have 
$$
\Gcal_\lambda(x)=\big(\nabla^2\omega(x)\big)^{-1}\nabla T_{\lambda}(x),$$ 
which can be viewed as a rescaled gradient of the Bregman Moreau envelope. Then with the assumption that $ \omega(x) $ is $1$-strongly convex with respect to $\|\cdot\|$-norm, we have $\|\Gcal_\lambda(x)\|\leq \| \nabla T_{\lambda}(x)\|$.  Moreover, from the definition of Bregman proximal operator $\hat{x}=\prox_{\lambda T}(x)$, we have
\begin{equation}
0\in\partial (T+\delta_X)(\hat{x})+\frac{1}{\lambda}\big(\nabla\omega(\hat{x})-\nabla\omega(x)\big)\approx\partial (T+\delta_X)(\hat{x})+\nabla^2\omega(x)\Gcal_\lambda(x)
\end{equation}
where the approximation is based on the first-order Taylor expansion of $ \nabla\omega(\cdot) $. 
Hence when $\|\Gcal_\lambda(x)\|$ is small, it indicates that the origin is near the set $\partial(T+\delta_X)(x)$, \ie, $\hat{x}$ is close to a stationary point.

To better capture the geometry of the non-Euclidean setup, we propose to measure the stationarity of a candidate solution through evaluating the Bregman divergence the solution and its proximity: 
\beq{eq:Bregman_measure2}
\Delta_\lambda(x):=\frac{1}{\lambda^2}\cdot\left(D_\omega(x,\prox_{\lambda T}(x))+D_\omega(\prox_{\lambda T}(x),x)\right).
\eeq 
We call this the \emph{Bregman stationarity} measure. It follows immediately from the $1$-strongly convexity of $\omega(\cdot)$ that $\|\Gcal_\lambda(x)\|^2\leq\Delta_\lambda(x)$. Hence, the measure $\Delta_\lambda(x)$ yields a stronger convergence criterion than using the squared norm of the Bregman gradient mapping. Further, suppose the distance generating function $\omega(x)$ has $M$-Lipschitz continuous gradient, then we have
\beq{eq:distance_general}
\text{dist}^2_{\|\cdot\|}(0,\partial (T+\delta_X)(\hat{x}))\leq \frac{1}{\lambda^2}\|\nabla\omega(x)-\nabla\omega(\hat{x}) \|^2\leq \frac{M}{\lambda^2}\langle\nabla\omega(x)-\nabla\omega(\hat{x}),x-\hat{x}\rangle=M\cdot\Delta_\lambda(x).
\eeq
Hence, the measure defined by $\Delta_\lambda(x)$ provides a valid characterization of  the stationarity of a candidate solution in terms of the norm $\|\cdot\|$. In particular, for the $\ell_1$-setup with  $\|\cdot\|=\|\cdot\|_1$, the squared distance defined by $\ell_1$-norm in \eq{eq:distance_general} could be of order $\Ocal(n)$ larger than that defined by $\ell_2$-norm in \eq{eq:distance}. 

\section{Stationary Convergence of Stochastic Mirror Descent (SMD)}
\label{sec:main}
 In this section, we formally describe the problem setting and assumptions, and revisit the stochastic mirror descent algorithm. We will then discuss its convergence behavior in terms of the previously defined stationarity measure. 

\subsection{Problem Setting and Assumptions}
We consider the general composite stochastic optimization problem:
\begin{equation}\label{eq:obj2}
\min_{x\in{X}}\;T(x):= f(x)+r(x)=\EE_{\xi}\left[F(x;\xi)\right]+r(x)
\end{equation} 
under the following assumptions:
\begin{assume}\label{AssumeSMD3}
	We assume that
	\begin{enumerate}[label=(\roman*),labelindent=1em,labelwidth=\widthof{\ref{last-item}},leftmargin=!]
		\item The set $ {X}\subseteq\mathcal{X} $ is a closed convex subset of a finite-dimensional Euclidean space $ \mathcal{X} $.
		\item The function $f(x)$ is $(\rho,\omega(\cdot))$-RWC on $X$, for some function $ \omega(\cdot) $ that is continuously differentiable and  1-strongly convex (1-SC) on $X$ with respect to the norm $\|\cdot\|$ defined on the  Euclidean space $ \mathcal{X} $.  
		\item There exists a stochastic oracle that outputs a random vector $G(x,\xi)$ given input $ x\in X $,  such that
		\begin{equation}
		\mathbb{E}_{\xi}\Big[G(x,\xi)\Big]\in\partial f(x)
		\end{equation}
		where $\partial f(x)$ is the subdifferential set of $f(x)$ at $x$. Moreover, we assume there exists a constant $L>0$, such that $\forall x\in X$
		\beq{eq:variance}
		\EE[\|G(x,\xi)\|_*^2]\leq L^2;
		\eeq
		This is sometimes called $L$-stochastically continuity of $f(x)$~\citep{lu2017relative}.
		\item The function $ r(x):X\to\RR $ is proper, closed, convex, {nonnegative} and perhaps nonsmooth.
		\item The  optimal objective value, denoted as $ T_{\min}$, exists and $T_{\min}>-\infty $.
	\end{enumerate}
\end{assume}

Note that here we assume the term $ r(\cdot) $ to be nonnegative, which is a common assumption for proximal algorithms in the literature; see e.g., \cite{duchi2010composite} and \cite{beck17firstorder}. This assumption is also satisfied by a wide range of regularizations used in practical applications.

\subsection{Stochastic Mirror Descent (SMD)}
We now formally present the SMD algorithm as outlined in Algorithm~\ref{SMD2} below. Here we are going to use $\omega(x)$ as the distance generating function for the Bregman divergence used in the SMD algorithm. For the sake of generality, we will adopt the proximal variant of SMD, which has been extensively studied for convex problems; see, e.g.,~\cite{duchi2010composite},~\cite{he2015saddle}, and~\cite{beck17firstorder}. The only modifications we make is that when generating an output solution after $N$ iterations, we will randomly pick one from the sequence $\{x_0,x_1,\ldots,x_{N-1}\}$ according to a fixed distribution based on the stepsizes. 

\begin{algorithm} 
	\caption{Stochastic Mirror Descent (SMD)}
	\label{SMD2}  
	\begin{algorithmic} 
		\STATE Input { $x_0,  N,\{\alpha_t\}_{t=0}^{N-1}$}
		\FOR {$ t=0 $ to $ N-1 $}
		\STATE Obtain $G_t\coloneqq G(x_t,\xi_t)$ from the stochastic oracle
		\STATE Update $ x_{t+1}
		=\argmin_{x\in X}
		\left\{\langle G_t, x \rangle+ r(x)+\frac{1}{\alpha_t}D_{\omega}(x,x_t)\right\} $
		\ENDFOR
		\STATE Output $ x_R $ from $ \{x_0, \ldots, x_{N-1}\} $ with probability as $ P(R=i)=\frac{\alpha_i}{\sum_{t=0}^{N-1}\alpha_t},\ (i=0,1,\cdots,N-1) $ 
	\end{algorithmic}  
\end{algorithm}
We emphasize that the SMD algorithm significantly differs from the RSPG algorithm proposed in \cite{ghadimi2016mini} in several aspects: first, we don't need to use mini-batch samples to construct the subgradient estimator; second, the stepsize has to be decaying or in the order of $\Ocal(1/\sqrt{N})$ rather than a large constant; third, the probability mass function for selecting a random output is much simpler.  

\subsection{Convergence Results}

Below we present the stationary convergence result of SMD. 
\begin{thm}\label{thm:main}
	\emph{(Stationary Convergence of SMD)}
	Let $ x_R $ be the output of the SMD algorithm after $N$ iterations with non-increasing stepsize $ \alpha_t>0, t=0,\ldots, N-1$. Then we have for any $\hat\rho$ such that $\hat\rho>\rho$, 
\begin{equation}\label{SMDcase3Result}
	\mathbb{E}[\Delta_{1/\hat{\rho}}(x_R)]\leq\frac{\hat{\rho}}{\hat{\rho}-\rho}\cdot\frac{T_{1/\hat{\rho}}(x_0)-T_{\min}+\hat{\rho}\alpha_0 r(x_0)+\frac{\hat{\rho}L^2}{2}\sum_{t=0}^{N-1}\alpha_t^2}{\sum_{t=0}^{N-1}\alpha_t},
	\end{equation}
	where $ \Delta_{1/\hat{\rho}}(x_R)$ is as defined in~\eq{eq:Bregman_measure2} and the expectation is taken with respect to $R$ and $(\xi_0,\cdots,\xi_{N-1})$. 
\end{thm}

The above theorem provides the first characterization of the non-asymptotic convergence behavior of the SMD algorithm in expectation. As discussed in previous section, the stationary measure $\Delta_{1/\hat{\rho}}(x)$ with $\hat{\rho}>\rho$ provides a meaningful way to evaluate the stationarity of a candidate solution and also captures the underlying geometry of the non-Euclidean setup. It is worth mentioning that this result  generalizes the recent convergence results~\citep{davis2018stochastic} for  stochastic projected subgradient method and stochastic proximal  subgradient method in a unified sense. In~\cite{davis2018stochastic}, the authors develop two different results and analysis for the projected and proximal versions of stochastic subgradient method. For the proximal version, their convergence result requires $\hat\rho\in(\rho,2\rho]$ and the stepsize $\alpha_t\leq 1/\hat{\rho}$ for algebraic purposes. However, such requirements are not needed in our analysis.

In particular, if we select the stepsize to be a constant and set $\hat\rho=2\rho$, our result yields

\begin{cor}\label{cor:rate}
	For a fixed number of iterations $N$, by setting the stepsize to be a constant $\alpha_t\equiv\frac{c}{\sqrt{N}}, t=0,1,\ldots, N-1$, for some positive scaler $c>0$, the solution $x_R$ generated by the SMD algorithm satisfies 
	\begin{equation}\label{SMDCase4Result}
	\mathbb{E}[\Delta_{1/(2\rho)}(x_R)]\leq2\cdot\Big(\frac{T_{1/(2\rho)}(x_0)-T_{\min}+\rho c^2L^2}{c\sqrt{N}}+\frac{r(x_0)}{N}\Big).
	\end{equation}
	\end{cor}
We can further optimize the choice of stepsize and obtain
\begin{cor}\label{cor:rate2}
Suppose that $T_{\min}$ is known and assume that we can initialize SMD with $x_0$ such that $r(x_0)=0$. Then by setting the stepsize to be $\alpha_t\equiv\frac{c}{\sqrt{N}},t=0,1,\ldots,N-1 $, such that
	\begin{equation}
	c=\sqrt{\frac{T_{1/(2\rho)}(x_0)-T_{\min}}{\rho L^2}},
	\end{equation}
we further have
	\begin{equation}
	\mathbb{E}[||\Gcal_{1/(2\rho)}(x_R)||^2]\leq \mathbb{E}[\Delta_{1/(2\rho)}(x_R)]\leq 
	\frac{4L\sqrt{\rho\big(T_{1/(2\rho)}(x_0)-T_{\min}\big)}}{\sqrt{N}}
	\end{equation}
\end{cor}

The above corollaries imply that the SMD algorithm converges to a stationary point in the rate of $\Ocal(1/\sqrt{N})$. In other words, to obtain an $\eps$-stationary solution such that $\EE[\Delta_{1/(2\rho)}(x_R)]\leq\epsilon$, the iteration complexity and sample complexity for SMD is at most $O\left(\frac{\rho L^2(T_{1/(2\rho)}(x_0)-T_{\min})}{\eps^2}\right)$. The order of sample complexity, \ie, $\Ocal(1/\eps^2)$ matches with that of existing algorithms, such as the RSPG algorithm~\citep{ghadimi2016mini}, the PGSG algorithm~\citep{davis2017proximally}, and the proximal stochastic subgradient algorithm~\citep{davis2018stochastic} for solving nonsmooth nonconvex optimization. 

\subsection{Convergence Analysis}
In this section, we provide the detailed proof for Theorem~\ref{thm:main}. \\
\begin{proof}
For sake of simplicity, in what follows, we will denote $\hat x:=\prox_{T/\hat{\rho}}(x)$ for any $x\in X$. First, by the definition of Bregman envelope, we have $T_{1/\hat{\rho}}(x_{t+1})=T(\hat{x}_{t+1})+\hat{\rho}D_{\omega}(\hat{x}_{t+1}, x_{t+1}) $. The optimality of $ \hat{x}_{t+1} $ implies
	\begin{equation}\label{eq:optimality1}
	T_{1/\hat{\rho}}(x_{t+1})\leq T(\hat{x}_t)+\hat{\rho}D_{\omega}(\hat{x}_t, x_{t+1}).
	\end{equation}
Recall the definition of $ x_{t+1} $ and apply the three-point property introduced in Lemma~\ref{lem:threepoint}(b) and equation \eq{threeProI} by setting $ z=x_t,\ z^+=x_{t+1},\ x=\hat{x}_t$ and $\alpha=\alpha_t$, $\phi(x)=\la G_t,x\ra+r(x)$. We have
	\begin{equation}\label{eq:optimality2}
	\alpha_t[\langle G_t, \hat{x}_t-x_{t+1}\rangle+r(\hat{x}_t)-r(x_{t+1})]\geq D_\omega(\hat{x}_t,x_{t+1})+D_{\omega}(x_{t+1},x_t)-D_{\omega}(\hat{x}_t, x_t)
	\end{equation}
Combing equations~\eq{eq:optimality1} and~\eq{eq:optimality2}, we have
	\begin{equation}\label{eq:decompose}
	\begin{split}
	\ &\mathbb{E}\Big[T_{1/\hat{\rho}}(x_{t+1})\Big]\\
	\leq\ &
	\mathbb{E}\Big[T(\hat{x}_t)+\hat{\rho}\alpha_t\langle G_t, \hat{x}_t-x_{t+1}\rangle+\hat{\rho}\alpha_t\big(r(\hat{x}_t)-r(x_{t+1})\big)+\hat{\rho}D_{\omega}(\hat{x}_t, x_t)-\hat{\rho}D_{\omega}(x_{t+1},x_t)\Big]\\
	=\ &
	\mathbb{E}\Big[T_{1/\hat{\rho}}(x_t)+\hat{\rho}\alpha_t\langle G_t, \hat{x}_t-x_{t+1}\rangle+\hat{\rho}\alpha_t\big(r(\hat{x}_t)-r(x_{t+1})\big)-\hat{\rho}D_{\omega}(x_{t+1},x_t)\Big]\\
	=\ &
	\mathbb{E}\Big[T_{1/\hat{\rho}}(x_t)\Big]+ \hat{\rho}\alpha_t\mathbb{E}\Big[\langle G_t, \hat{x}_t-x_t\rangle+\big(r(\hat{x}_t)-r(x_t)\big)\Big]+\hat{\rho}\mathbb{E}\Big[\alpha_t\big(r(x_t)-r(x_{t+1})\big)\Big]\\
	&\qquad\qquad\quad +\hat{\rho}\mathbb{E}\Big[\alpha_t\langle G_t, x_t-x_{t+1}\rangle-D_{\omega}(x_{t+1},x_t)\Big]
	\end{split}
	\end{equation}
	where the first equality comes from the definition of $T_{1/\hat{\rho}}(x_t)$. 

	Next, invoking the $(\rho,\omega(\cdot))$-relatively weakly convexity of the function $f(x)$, we have 
	\begin{equation}
	\EE[\langle G_t, \hat{x}_t-x_t\rangle]
	\leq f(\hat{x}_t)-f(x_t)+\rho D_\omega(\hat{x}_t,x_t)
	\end{equation}
	where the expectation is taking over $\xi_t|\xi_0,\ldots,\xi_{t-1}$. Combine with $ r(x) $, this implies that the second term in equation \eq{eq:decompose} can be bounded by 
	\begin{equation}\label{eq:second_term}
		\EE\Big[\langle G_t, \hat{x}_t-x_t\rangle+r(\hat{x}_t)-r(x_t)\Big]\leq T(\hat{x}_t)-T(x_t)+\rho D_\omega(\hat{x}_t,x_t)
	\end{equation} 
	Moreover, it is easy to see that the last term in equation \eq{eq:decompose} can also be bounded as follows, 
	\begin{align}\label{eq:last_term}
	\hat{\rho}\mathbb{E}\Big[\alpha_t\langle G_t, x_t-x_{t+1}\rangle-D_{\omega}(x_{t+1},x_t)\Big] 
	\leq &\hat{\rho}\mathbb{E}\Big[\alpha_t\langle G_t, x_t-x_{t+1}\rangle-\frac{1}{2}||x_{t+1}-x_t||^2\Big]\nonumber\\
	 \leq &\frac{1}{2}\hat{\rho}\alpha_t^2\cdot\EE[||G_t||_*^2] \leq\ \frac{1}{2}\hat{\rho}\alpha_t^2L^2
	\end{align}
	Here the first inequality is due to the fact that $D_\omega(x,y)\geq\frac{1}{2}\|x-y\|^2$ and the second inequality is due to Young's inequality. Hence, combining \eq{eq:decompose} with \eq{eq:second_term} and \eq{eq:last_term}, we end up with 
	\begin{equation}
	\begin{split}
		&\mathbb{E}\Big[T_{1/\hat{\rho}}(x_{t+1})\Big]\\
		\leq\ &
		\mathbb{E}\Big[T_{1/\hat{\rho}}(x_t)+\hat{\rho}\alpha_t\big(T(\hat{x}_t)-T(x_t)+\rho D_\omega(\hat{x}_t,x_t)\big)+\hat{\rho}\alpha_t\big(r(x_t)-r(x_{t+1})\big)+\frac{\hat{\rho}\alpha_t^2L^2}{2}\Big]
	\end{split}
	\end{equation}
	Therefore, by telescoping the sum from $t=0$ to $N-1$, and moving terms around, we further arrive at 
    \begin{eqnarray}
	&&\sum_{t=0}^{N-1}\mathbb{E}\Big[\alpha_t\big(T(x_t)-T(\hat{x}_t)-\rho D_\omega(\hat{x}_t,x_t)\big)\Big]\\
	&\leq&\frac{1}{\hat{\rho}}\big(T_{1/\hat{\rho}}(x_0)-T_{1/\hat{\rho}}(x_N)\big)+\sum_{t=0}^{N-1}\alpha_t\big(r(x_t)-r(x_{t+1})\big)+\frac{L^2}{2}\sum_{t=0}^{N-1}\alpha_t^2\\
	&\leq &\frac{1}{\hat{\rho}}\big(T_{1/\hat{\rho}}(x_0)-T_{\min}\big)+\alpha_0r(x_0)+\frac{L^2}{2}\sum_{t=0}^{N-1}\alpha_t^2\label{eq:total}
	\end{eqnarray}
	The last inequality is because that the stepsize $\alpha_t$ is non-increasing and the function $r(x)$ is nonnegative, which leads to 
	\begin{equation*}
		\begin{split}
		\sum_{t=0}^{N-1}\alpha_t\big(r(x_t)-r(x_{t+1})\big)
		=\ &
		\alpha_0r(x_0)-\alpha_{N-1}r(x_N)+\sum_{t=0}^{N-2}(\alpha_{t+1}-\alpha_t)r(x_{t+1})
		\leq
		\alpha_0r(x_0).
		\end{split}
	\end{equation*}

	Finally, let us divide both sides of equation \eq{eq:total} by $ \sum_{t=0}^{N-1}\alpha_t $, we finally obtain
	\begin{equation}\label{eq:mainresult}
	\frac{\sum_{t=0}^{N-1}\mathbb{E}\Big[\alpha_t\big(T(x_t)-T(\hat{x}_t)-\rho D_\omega(\hat{x}_t,x_t)\big)\Big]}{\sum_{t=0}^{N-1}\alpha_t}
	\leq
	\frac{T_{1/\hat{\rho}}(x_0)-T_{\min}+\hat{\rho}\alpha_t r(x_0)+\rho D_\omega(\hat{x}_t,x_t)}{\hat{\rho}\sum_{t=0}^{N-1}\alpha_t}.
	\end{equation}
	Invoking the definition of $ x_R $ in the algorithm, this implies that 
	\begin{equation}
	\text{LHS}=\mathbb{E}\Big[T(x_R)-T(\hat{x}_R)-\rho D_\omega(\hat{x}_R,x_R)\Big].
	\end{equation}
	Recall that $\hat{x}_R$ is the minimizer of the problem, $ \argmin_{x\in X}\;\{T(x)+\hat{\rho}D_{\omega}(x,x_R)\}$, and the objective is $(\hat\rho-\rho)$-relatively strongly convex. It follows from Lemma~\ref{lem:threepoint} that
	 \begin{equation}\label{eq:sc}
	T(x_R)-[T(\hat{x}_R)+\hat{\rho}D_\omega(\hat{x}_R,x_R)]\geq (\hat{\rho}-\rho)D_\omega(x_R,\hat{x}_R)
	\end{equation}
	Hence, we can further derive that 
	\begin{equation}
	\begin{split}
	\text{LHS}=\ &\mathbb{E}\Big[T(x_R)-T(\hat{x}_R)-\rho D_{\omega}(\hat{x}_R,x_R)\Big]\\
	=\ &\mathbb{E}\Big[T(x_R)-\big[T(\hat{x}_R)+\hat{\rho}D_{\omega}(\hat{x}_R,x_R)\big]+(\hat{\rho}-\rho) D_{\omega}(\hat{x}_R,x_R)\Big]\\
	\geq\ 
	&\mathbb{E}\Big[(\hat{\rho}-\rho)D(x_R,\hat{x}_R)+(\hat{\rho}-\rho) D_{\omega}(\hat{x}_R,x_R)\Big]\\
	=\
	&\frac{(\hat{\rho}-\rho)^2}{\hat\rho^2}\mathbb{E}\Big[\Delta_{1/\hat{\rho}}(x_R)\big].
	\end{split}
	\end{equation}
	Here second inequality is from \eq{eq:sc} and the last inequality is simply using the definition of $\Delta_{1/\hat{\rho}}(x_R)$. Combining with equation \eq{eq:mainresult}, we arrive at the desired result as stated in the theorem.
	\end{proof}

\section{Extension to Relatively Continuous Nonconvex Problems}
\label{sec:extensions}
In this section, we further extend the previous stationary convergence results of SMD under relaxed assumptions of the Lipschitz continuity of the function $f(x)$. A standard condition for applying the SMD algorithm to stochastic nonsmooth problems is to assume that the stochastic gradient has bounded moments, \ie, $\max_{x\in X}\;\EE[\|G(x,\xi)\|_*^2]\leq L^2$. Recent works (e.g., \cite{lu2017relative}) show that such an assumption is not always satisfied in practice, particularly for those objectives without Lipschitz continuity. Here we generalize the convergence results to a broader class of nonsmooth nonconvex functions that are possibly non-Lipschitz continuous. 

 Let $\omega(x):X\to\RR$ be a reference function that is differentiable and 1-strongly convex on ${X}$ with respect to the given norm $\|\cdot\|$ and $D_\omega(x,y)$ be the Bregman divergence induced by $\omega(x)$.  
\begin{defn}
	\emph{(Stochastically (Fr\'echet) Relatively Continuous Functions)}  

	
	A function $ f(x) $ is called \emph{$ L $-Stochastically relatively continuous} with respect to $ \omega(x) $ on a set $ {X} $, denoted as $ (L,\omega(\cdot))$-SRC,  for some positive constant $L>0$,  if for any $x\in X$ and any unbiased estimator $G(x,\xi)$ of the subgradient of $ f(\cdot) $ at $ x $, satisfy $\EE[G(x,\xi)]\in\partial f(x)$, and  
		\begin{equation}
	\mathbb{E}\big[||G(x,\xi)||^2_*\big]\leq\frac{L^2D_{\omega}(y,x)}{\frac{1}{2}||y-x||^2}, \forall y\neq x.
	\end{equation}

\end{defn}

\begin{lm}\label{lem:SRC_property}
	\emph{(Binomial Property of SRC Functions, \cite{lu2017relative})} Let $f(x)$ be a $(L,\omega(\cdot))$-SRC function and $ x\in{X}$. Define the random vector
	\begin{equation}
	M(x,\xi)\coloneqq||G(x,\xi)||_*\cdot\max_{y\in{X},y\neq x}\frac{||y-x||}{\sqrt{2D_{\omega}(y,x)}}.
	\end{equation}
	Then it holds that 
	\begin{enumerate}
	\item[(a)] $ \mathbb{E}\big[M^2(x,\xi)\big]\leq L^2 $, and
	\item[(b)] For any $ \alpha>0 $,  $\langle \alpha G(x,\xi),x-y\rangle-D_{\omega}(y,x)\leq \frac{1}{2}\alpha^2 M^2(x,\xi)$.
\end{enumerate}
\end{lm}

\begin{thm}\label{SMDcase2Thm}
Suppose that $f(x)$ is $(\rho,\omega(\cdot))$-RWC and $(L,\omega(\cdot))$-SRC as defined. Let $ x_R $ be the output of SMD algorithm for solving the problem \eq{eq:obj} within a fixed iteration number $ N>0 $, and constant stepsize $ \alpha_t=c/\sqrt{N}$, where $c>0$.  We have 
\begin{equation}\label{eq:final_result}
	\mathbb{E}[\Delta_{1/(2\rho)}(x_R)]\leq2\cdot\Big(\frac{T_{1/(2\rho)}(x_0)-T_{\min}+\rho c^2L^2}{c\sqrt{N}}+\frac{r(x_0)}{N}\Big).
	\end{equation}
\end{thm}

\begin{proof}
The theorem can be proved with a slight modification of the proof as detailed in the previous section. When the function $T(x)$ is $(L,\omega(\cdot))$-SRC, we can still prove the equation~\eq{eq:last_term} by directly apply Lemma~\ref{lem:SRC_property}.
\end{proof}

\begin{remark}
Note that the reference function $ \omega(\cdot) $ used to define either the relatively weak convexity or the stochastically relative continuity is the same as the distance generating function used in the SMD algorithm as well as in the Bregman Moreau envelope. 
\end{remark}

\begin{remark} \emph{(Deterministic Setting)}
The results developed in Sections~\ref{sec:main} and ~\ref{sec:extensions} also apply to  deterministic nonconvex problems and the deterministic Mirror Descent algorithm. Particularly, suppose we have access to a subgradient oracle that returns $g(x)\in\partial f(x)$ for any input $x$, and $||g(x)||^2_*\leq\frac{L^2D_{\omega}(y,x)}{\frac{1}{2}||y-x||^2}, \forall y\neq x.$ Suppose the Mirror Descent algorithm performs the updates:
$$ x_{t+1}=\argmin_{x\in X}\left\{\langle g(x_t), x \rangle+ r(x)+\frac{1}{\alpha_t}D_{\omega}(x,x_t)\right\}$$ for $t=0,1,\ldots, N-1$, and outputs a solution $x_R$ such that $x_R=\argmin_{t=0,\ldots,N-1}\{\Delta_{1/(2\rho)}(x_t)\}$. Then with constant stepsize $ \alpha_t=c/\sqrt{N}$, where $c>0$, we have 
\begin{equation}\label{eq:final_result}
	\Delta_{1/(2\rho)}(x_R)\leq2\cdot\Big(\frac{T_{1/(2\rho)}(x_0)-T_{\min}+\rho c^2L^2}{c\sqrt{N}}+\frac{r(x_0)}{N}\Big).
	\end{equation}
\end{remark}
In other words, the number of subgradient evaluations needed to obtain an $\epsilon$-stationary solution such that $\Delta_{1/(2\rho)}(x)\leq\eps$, is at most $\Ocal(1/\eps^2)$. This result seems to be also the first non-asymptotic convergence result for the deterministic Mirror Descent algorithm and its proximal variant. 

\section{Conclusion}
In this paper, we establish the first non-asymptotic convergence analysis of Stochastic Mirror Descent (SMD) for solving a general class of nonconvex nonsmooth optimization problems, under relaxed conditions of weak convexity and continuity.  Our analysis applies to many variants in the family of SMD algorithms, and indicates that using mini-batch is not necessary for stationary convergence of SMD. We also show that using non-Euclidean setup could yield stronger stationarity guarantees.  For future work, we will investigate the convergence behaviors of other algorithms in the SMD family under different settings, both in theory and in real applications. 

\small
\singlespacing
\noindent

\bibliographystyle{plainnat}

\end{document}